\documentclass[amstex,10pt,reqno]{amsart}
\usepackage{amsmath,amsfonts,amssymb,amsthm,enumerate,hyperref,multicol}
\newtheorem{lemma}{Lemma}
\newtheorem{thm}{Theorem}

\newtheorem{remark}{Remark}
\numberwithin{equation}{section}
\begin{document}

\leftline{ \scriptsize \it}
\title[]
{Bivariate Generalized Baskakov Kantorovich Operators}
\maketitle

\begin{center}
{\bf Meenu Goyal$^1$ and P. N. Agrawal$^2$}
\vskip0.2in
$^{1,2}$Department of Mathematics\\
Indian Institute of Technology Roorkee\\
Roorkee-247667, India\\
\vskip0.2in

$^1$meenu.goyaliitr@yahoo.com
 $^2$pna\_{}iitr@yahoo.co.in
\end{center}
\begin{abstract}
This paper is in continuation of our work in \cite{PNM}, wherein we
introduced generalized Baskakov Kantorovich operators $K_n^a(f;x)$ and
established some approximation properties e.g. local approximation, weighted
approximation, simultaneous approximation and $A-$statistical convergence. Also,
we discussed the rate of convergence for functions having a derivative coinciding
a.e. with a function of bounded variation. The purpose of this paper is to study the
bivariate extension of the operators $K_n^a(f;x)$ and discuss results on the degree
of approximation, Voronovskaja type theorems and their
first order derivatives in polynomial weighted spaces.
\end{abstract}

\section{introduction}
In \cite{ERN}, Erencin defined the Durrmeyer type modification of generalized
Baskakov operators introduced by Mihesan \cite{MSN}, as
\begin{eqnarray*}
L_n^a(f;x)=\sum_{k=0}^{\infty}W_{n,k}^{a}(x)\frac{1}{B(k+1,n)}\int_{0}^{\infty}
\frac{t^k}{(1+t)^{n+k+1}}f(t)dt,\,\,  x\geq 0,
\end{eqnarray*}
where\\
$W_{n,k}^{a}(x)=e^{\frac{-ax}{1+x}}\frac{P_{k}(n,a)}{k!}\frac{x^k}{(1+x)^{n+k}}
, P_k(n,a)=\displaystyle\sum_{i=0}^{k}{n\choose k}(n)_ia^{k-i},$
and $(n)_0=1, (n)_i=n(n+1)...(n+i-1)$ for $i\geq 1$
and discussed some approximation properties.
Very recently, Agrawal et al.
\cite{PVS} studied the simultaneous approximation and approximation of functions
having derivatives of bounded variation by these operators.\\
In \cite{PNM}, we proposed the Kantorovich modification of
 generalized Baskakov  operators for the function $f$ defined on
$C_{\gamma}[0,\infty):= \{f\in C[0,\infty): f(t) = O(t^{\gamma})$ as
$t\rightarrow\infty,$ for some $\gamma>0\}$ as follows :
\begin{eqnarray}\label{e1}
K_n^a(f;x)=(n+1)\displaystyle\sum_{k=0}^{\infty}W_{n,k}^{a}(x)\displaystyle\int_{\frac{k}{n+1}}^{\frac{k+1}
{n+1}}f(t)dt,\,\,a\geq 0.
\end{eqnarray}
In the present paper, our aim is to consider the bivariate extension of these operators.
In this direction, Stancu \cite{DDS} first introduced new linear positive operators in two
and several dimensional variables. After that Barbosu \cite{BAR} studied the bivariate
extension of Stancu generalization of $q-$Bernstein operators. Dogru and Gupta \cite{DV}
constructed a bivariate generalization of the Meyer-Konig and Zeller operators based on
$q-$ integers while in \cite{OA} Agratini presented two-dimensional extension of some univariate
positive approximation processes expressed by series.

Many papers were published on approximation by modified Sz$\acute{a}$sz-Mirakyan and
Baskakov operators for functions of one or two variables cf. \cite{ALT1, ALT2, BFR, MGR, LRM, MSK, YOU} which deal with convergence, degree of approximation and Voronovskaja type theorems as well as convergence of partial derivatives of these
operators. Wafi and Khatoon \cite{WAF} defined the generalized Baskakov operators for
functions of two variables in polynomial and exponential weighted spaces and discussed
the rate of convergence and direct results. Later, in \cite{AWS}, the convergence of first order derivatives of these operators and a Voronovskaja type theorem were studied.
Recently, $\ddot{O}$rkc$\ddot{u}$ \cite{MO} established the approximation
properties of bivariate extension of q-Sz$\acute{a}$sz-Mirakjan-Kantorovich operators.
Very recently, Agrawal et al. \cite{PZF} constructed a bivariate extension of the Bernstein-Schurer-Kantorovich operators and discussed the rate of convergence and the asymptotic approximation. \\

\section{Construction of the operators}
Let $I=[0,\infty)\times[0,\infty)$ and $\omega_\gamma(x)
=(1+x^{\gamma})^{-1}$ for $\gamma\in \mathbb{N}_0$(set of all non-negative integers).\\
Further, for fixed $\gamma_1,\gamma_2\in \mathbb{N}_0,$ let $\omega_{\gamma_1,\gamma_2}
(x,y)=\omega_{\gamma_1}(x)\omega_{\gamma_2}(y).$
Then, for $f\in C_{\gamma_1,\gamma_2}(I)
:= \{f\in C(I): \omega_{\gamma_1,\gamma_2}(x,y)f(x,y)\,\, \mbox{is bounded and uniformly continuous on I}\},$
we define a bivariate extension of the operators (\ref{e1}) as follows:
\begin{eqnarray}\label{e8}
K_{n_1,n_2}^a(f;x,y)=(n_1+1)(n_2+1)\sum_{k_{1}=0}^{\infty}\sum_{k_2=0}^{\infty}
W_{n_1,n_2,k_1,k_2}^a(x,y)\int_{\frac{k_2}{n_2+1}}^{\frac{k_2+1}{n_2+1}}
\int_{\frac{k_1}{n_1+1}}^{\frac{k_1+1}{n_1+1}} f(u,v)du\,dv,
\end{eqnarray}
where
\begin{eqnarray*}
W_{n_1,n_2,k_1,k_2}^a(x,y)=\frac{x^{k_1}y^{k_2}p_{k_1}(n,a)p_{k_2}(n,a)
e^{\frac{-ax}{1+x}}e^{\frac{-ay}{1+y}}}{k_1!k_2!(1+x)^{n_1+k_1}(1+y)^{n_2+k_2}}.
\end{eqnarray*}
If $f\in C_{\gamma_1,\gamma_2}(I)$ and if $f(x,y)=f_1(x)f_2(y)$ for all $(x,y)\in I,$ then
\begin{eqnarray}\label{p1}
K_{n_1,n_2}^a(f(u,v);x,y)=K_{n_1}^a(f_1(u);x)K_{n_2}^a(f_2(v);y),
\end{eqnarray}
for $(x,y)\in I$ and $n_1, n_2\in \mathbb{N}.$ The sup norm on $C_{\gamma_1,\gamma_2}(I)$ is given by
\begin{eqnarray}\label{m1}
\parallel f\parallel_{\gamma_1,\gamma_2}=\sup_{(x,y)\in I}|f(x,y)|\omega_{\gamma_1,\gamma_2}(x,y), f\in C_{\gamma_1,\gamma_2}(I).
\end{eqnarray}
For $f\in C_{\gamma_1,\gamma_2}(I),$ we define the modulus of continuity
\begin{eqnarray}\label{m2}
\omega (f,C_{\gamma_1,\gamma_2};t,s) := \displaystyle\sup_{0<h<t, 0< \delta<s} \parallel \Delta_{h,\delta}f(.,.)\parallel_{\gamma_1,\gamma_2}, \,\,\, t,s \geq 0,
\end{eqnarray}
where $\Delta_{h,\delta}f(x,y):= f(x+h,y+\delta)-f(x,y)$ for $(x,y)\in I$ and $h,\delta>0.$\\
Moreover, for fixed $m\in \mathbb{N},$
let $C_{\gamma_1,\gamma_2}^m(I)$ be the space of all functions
$f\in C_{\gamma_1,\gamma_2}(I)$ having the partial derivatives $\dfrac{\partial^k}
{\partial x^s\partial y^{k-s}}\in C_{\gamma_1,\gamma_2}(I), s=1,2,...,k ;\, k=1,2,...,m.$
The paper is organized as follows:\\
In Section 2 of this paper, we give some definitions and auxiliary results
and in Section 3, we prove the main results of this paper wherein we study the degree of approximation, Voronovskaja type theorems and simultaneous approximation of first order derivatives for the operators $K_{n_1,n_2}^a$.

\section{Auxiliary results}
\begin{lemma}\label{lm1}
In \cite{PNM}, For the $r$th order central moment of $K_n^{a},$ defined as
\begin{eqnarray*}
u_{n,r}^a(x):= K_n^a((t-x)^r;x),
\end{eqnarray*}
we have
\begin{enumerate}[(i)]
\item  $u_{n,0}^a(x)=1, u_{n,1}^a(x)=\dfrac{1}{n+1}\bigg(-x+\dfrac{ax}{1+x}+
\dfrac{1}{2}\bigg)$ \\and $u_{n,2}^a(x)=\dfrac{1}{(n+1)^2}\bigg\{(n+1)x^2+(n-1)x
+\dfrac{a^2x^2}{(1+x)^2}+2ax\left(\dfrac{1-x}{1+x}\right)+\dfrac{1}{3}\bigg\};$\\ \item
$u_{n,r}^a(x)$ is a rational function of $x$ depending on the parameters a and r;\\ \item
for each $x\in (0,\infty), u_{n,r}^a(x)=O\bigg(\dfrac{1}{n^{[\frac{r+1}{2}]}}\bigg),$ where $[\beta]$ denotes the integer part of $\beta.$
\end{enumerate}
\end{lemma}

\begin{lemma}\label{l3}
Let $e_{i,j}:I\rightarrow I, e_{i,j}=x^i y^j, 0\leq i,j \leq 2$ be two-dimensional test
functions. Then the bivariate operators defined in (\ref{e8}) satisfy
the following results:
\begin{enumerate}[(i)]
\item $K_{n_1,n_2}^a(e_{0,0};x,y)=1;$
\item $K_{n_1,n_2}^a(e_{1,0};x,y)=\dfrac{1}{n_1+1}\bigg(n_1 x
+\frac{ax}{1+x}+\frac{1}{2}\bigg);$
\item $K_{n_1,n_2}^a(e_{0,1};x,y)=\dfrac{1}{n_2+1}\bigg
(n_2 y+\frac{ay}{1+y}+\frac{1}{2}\bigg);$
\item $K_{n_1,n_2}^a(e_{2,0};x,y)=\dfrac{1}{(n_1+1)^2}\bigg(n_1^2x^2+n_1x^2+2n_1x
    +\dfrac{2an_1x^2}{1+x}+\dfrac{a^2x^2}{(1+x)^2}+\dfrac{2ax}{1+x}+\frac{1}{3}\bigg);$
\item $K_{n_1,n_2}^a(e_{0,2};x,y)=\dfrac{1}{(n_2+1)^2}\bigg(n_2^2y^2+n_2y^2+2n_2y
    +\dfrac{2an_2y^2}{1+y}+\frac{a^2y^2}{(1+y)^2}+\frac{2ay}{1+y}+\frac{1}{3}\bigg);$
\item $K_{n_1,n_2}^a(e_{3,0};x,y)= \dfrac{1}{(n_1+1)^3}\bigg(n_1^3 x^3+\dfrac{3n_1^2x^2}{2}(3+2x)+\dfrac{3n_1x^2}{2}+\dfrac{n_1x}{2}(4x^2+6x+5)+\dfrac{3ax^3n_1^2}{1+x}$\\
    $+\dfrac{3an_1x^2}{1+x}\bigg\{(3+x)+\dfrac{ax}{1+x}\bigg\}+\dfrac{ax}{1+x}
    \bigg\{\dfrac{7}{2}+\dfrac{7}{2}\dfrac{ax}{(1+x)}+\dfrac{a^2x^2}{(1+x)^2}\bigg\}+\dfrac{1}{4}\bigg);$\\
\item $K_{n_1,n_2}^a(e_{0,3};x,y)= \dfrac{1}{(n_2+1)^3}\bigg(n_2^3 y^3+\dfrac{3n_2^2y^2}{2}(3+2y)+\dfrac{3n_2y^2}{2}+\dfrac{n_2y}{2}(4y^2+6y+5)+\dfrac{3ay^3n_2^2}{1+y}$\\
    $+\dfrac{3an_2y^2}{1+y}\bigg\{(3+y)+\dfrac{ay}{1+y}\bigg\}+\dfrac{ay}{1+y}
    \bigg\{\dfrac{7}{2}+\dfrac{7}{2}\dfrac{ay}{(1+y)}+\dfrac{a^2y^2}{(1+y)^2}\bigg\}+\dfrac{1}{4}\bigg).$
\end{enumerate}
\end{lemma}

\begin{lemma}\label{l4}
For $n_1, n_2\in \mathbb{N},$ we have
\begin{enumerate}[(i)]
\item $K_{n_1,n_2}^a(u-x;x,y)=\dfrac{1}{n_1+1}\bigg(-x+\dfrac{ax}{1+x}+\dfrac{1}{2}\bigg);$
\item $K_{n_1,n_2}^a(v-y;x,y)=\dfrac{1}{n_2+1}\bigg(-y+\dfrac{ay}{1+y}+\dfrac{1}{2}\bigg);$
\item $K_{n_1,n_2}^a((u-x)^2;x,y)=\dfrac{1}{(n_1+1)^2}\bigg((n_1+1)x^2+(n_1-1)x
+\dfrac{a^2x^2}{(1+x)^2}+2ax\left(\dfrac{1-x}{1+x}\right)+\dfrac{1}{3}\bigg);$
\item $K_{n_1,n_2}^a((v-y)^2;x,y)=\dfrac{1}{(n_2+1)^2}\bigg((n_2+1)y^2+(n_2-1)y
+\dfrac{a^2y^2}{(1+y)^2}+2ay\left(\dfrac{1-y}{1+y}\right)+\dfrac{1}{3}\bigg).$
\end{enumerate}
\end{lemma}

\begin{remark}\label{r1}
For every $x\in [0,\infty)$ and $n\in\mathbb{N},$ we have
\begin{eqnarray*}
K_{n_1,n_2}^a((u-x)^2;x,y)\leq \frac{\{\delta_{n_1}^a(x)\}^2}{n_1+1},
\end{eqnarray*}
where $ \{\delta_{n_1}^a(x)\}^2=\phi^2(x)+\frac{(1+a)^2}{n_1+1}$ and $\phi(x)=\sqrt{x(1+x)}.$
\end{remark}
\begin{proof}
From Lemma \ref{l4} (iii), we have
\begin{eqnarray*}
K_{n_1,n_2}^a((u-x)^2;x,y)&\leq&\frac{(n_1+1) x^2+n_1 x}{(n_1+1)^2}+\frac{1}{(n_1+1)^2}\bigg(\frac{a^2x^2}{(1+x)^2}+\frac{2ax}{1+x}+\frac{1}{3}\bigg)\\ &\leq&\frac{1}{n_1+1}\bigg(x(1+x)+\frac{(1+a)^2}{n_1+1}\bigg)=\frac{\{\delta_{n_1}^{a}(x)\}^2}{n_1+1}.
\end{eqnarray*}
\end{proof}

\begin{lemma}\label{l7}
For every $\gamma_1\in \mathbb{N}_0$ there exist positive constants $M_k(\gamma_1), k=1,2$ such that
\begin{enumerate}[(i)]
\item $\omega_{\gamma_1}(x) K_n^a\bigg(\dfrac{1}{\omega_{\gamma_1}(t)};x\bigg)\leq M_1(\gamma_1),$\\
\item $\omega_{\gamma_1}(x) K_n^a\bigg(\dfrac{(t-x)^2}{\omega_{\gamma_1}(t)};x\bigg)\leq M_2(\gamma_1)\dfrac{\{\delta_{n}^a(x)\}^2}{n+1},$
\end{enumerate}
for all $x\in \mathbb{R}_0=\mathbb{R}_+\cup \{0\},\mathbb{R}_+ = (0,\infty)$ and $n\in \mathbb{N}.$
\end{lemma}

\begin{lemma}\label{l5}
For every $\gamma_1,\gamma_2\in \mathbb{N}_0$ there exist positive constant $M_3(\gamma_1,\gamma_2),$ such that
\begin{eqnarray}\label{m0}
\parallel K_{n_1,n_2}^a(f;.,.)\parallel_{\gamma_1,\gamma_2}\leq M_3(\gamma_1,\gamma_2)\parallel f\parallel_{\gamma_1,\gamma_2}
\end{eqnarray}
for every $f\in C_{\gamma_1,\gamma_2}(I)$ and for all $n_1, n_2\in \mathbb{N}.$
\end{lemma}
\begin{proof}
From equation (\ref{p1}) and Lemma \ref{l7}, we get
\begin{eqnarray}\label{eq1}
\omega_{\gamma_1,\gamma_2}(x,y)K_{n_1,n_2}^a\bigg(\frac{1}{\omega_{\gamma_1,\gamma_2}(u,v)};x,y\bigg)&=&
\bigg(\omega_{\gamma_1}(x)K_{n_1}^a\bigg(\frac{1}{\omega_{\gamma_1}(u)};x\bigg)\bigg)
\bigg(\omega_{\gamma_2}(y)K_{n_2}^a\bigg(\frac{1}{\omega_{\gamma_2}(v)};y\bigg)\bigg)\nonumber\\
&\leq& \left\Vert K_{n_1}^a\bigg(\frac{1}{\omega_{\gamma_1}(u)};.\bigg)\right\Vert_{\gamma_1,\gamma_2} \times\left\Vert K_{n_2}^a\bigg(\frac{1}{\omega_{\gamma_1}(v)};.\bigg)\right\Vert_{\gamma_1,\gamma_2},
\end{eqnarray}
for all $(x,y)\in I$ and $n_1,n_2\in\mathbb{N}.$\\
Taking supremum on the left side of the inequality of (\ref{eq1}) and using (\ref{m1}), we obtain
\begin{eqnarray}\label{eq2}
\left\Vert K_{n_1,n_2}^a\bigg(\frac{1}{\omega_{\gamma_1,\gamma_2}(u,v)};.,.\bigg)\right\Vert_{\gamma_1,\gamma_2}\leq M_4(\gamma_1,\gamma_2).
\end{eqnarray}
Now,
\begin{eqnarray*}
\parallel K_{n_1,n_2}^a(f;.,.)\parallel_{\gamma_1,\gamma_2}&\leq& \parallel f \parallel_{\gamma_1,\gamma_2}
\left\Vert K_{n_1,n_2}^a\bigg(\frac{1}{\omega_{\gamma_1,\gamma_2}(u,v)};.,.\bigg)\right\Vert_{\gamma_1,\gamma_2}.
\end{eqnarray*}
From (\ref{eq2}), we get the desired result.
\end{proof}

\section{Main results}
\subsection{Local approximation}
For $f\in C_{B}(I)$ (the space of all bounded and uniformly continuous functions on $I,$
let $C_{B}^2(I) =\{f\in C_{B}(I):f^{(p,q)}\in C_{B}(I), \,1\leq p,q\leq 2\},$
where $f^{(p,q) }$ is $(p,q)$th-order partial derivative with respect to $x,y$ of $f,$ equipped with the norm
\begin{eqnarray*}
||f|| _{C_{B}^2(I)}=||f||_{C_{B}(I)}+\sum_{i=1}^2\left\| \frac{\partial^i f}{\partial x^i}\right\|+\sum_{i=1}^2\left\| \frac{\partial^i f}{\partial y^i}\right\|.
\end{eqnarray*}
The Peetre's $K-$functional of the function $f\in C_{B}(I)$ is given by
$$\mathcal{K}(f;\delta)=\inf_{g\in{C_{B}^2(I)}}\{||f-g||_{C_{B}(I)}+\delta||g||_{C_{B}^2(I)},\delta>0\}.$$
It is also known that the following inequality
\begin{eqnarray}\label{fr1}
\mathcal{K}(f;\delta)\leq M_{1}\{\omega_{2}(f;\sqrt{\delta })+\min(1,\delta )||f||_{C_B(I) }\}
\end{eqnarray}
holds for all $\delta>0$ (\cite{BB}, page 192). The constant $M_{1}$ is independent of $\delta $ and $f$ and $\omega_2(f;\sqrt{\delta})$ is the second order modulus of continuity which is defined in a similar manner as the second order modulus of continuity for one variable case:
\begin{eqnarray*}
\omega_2(f,\sqrt{\delta})=\displaystyle\sup_{0<h\leq\sqrt{\delta}}\displaystyle\sup_{x,x+2h\in [0,\infty)}|f(x+2h)-2f(x+h)+f(x)|.
\end{eqnarray*}
For $f\in C_B(I),$ the complete modulus of continuity for bivariate case is defined as follows:\\
$\omega(f;\delta)=\sup\bigg\{|f(u,v)-f(x,y)|:(u,v),(x,y)\in I$ and $\sqrt{(u-x)^2+(v-y)^2}\leq\delta\bigg\}.$\\
Further, the partial moduli of continuity  with respect to $x$ and $y$ is defined as\\
$\omega_{b}(f;\delta)=\sup\bigg\{|f(x_1,y)-f(x_2,y)|:y\in \mathbb{R}_0\,\ and\,\ |x_1-x_2|\leq\delta\bigg\},$\\
$\omega_{c}(f;\delta)=\sup\bigg\{|f(x,y_1)-f(x,y_2)|:x\in \mathbb{R}_0\,\ and\,\ |y_1-y_2|\leq\delta\bigg\}.$\\
It is clear that they satisfy the properties of the usual modulus of continuity.
The details of the modulus of continuity for the bivariate case can be found in \cite{AG}.\\

Now, we find the order of approximation of the sequence $K_{n_1,n_2}^a(f;x,y)$ to the function $f(x;y)\in C_{B}^2(I)$ by Petree's K-functional.

\begin{thm}\label{t11}
For the function $f\in C_B(I),$ the following inequality
\begin{eqnarray*}
|K_{n_1,n_2}^a(f;x,y)-f(x,y)|&\leq&4\mathcal{K}(f;M_{n_1,n_2}(x,y))\\&&+\omega\bigg(f;\sqrt{\bigg(\frac{1}{n_1+1}\bigg(-x+\frac{ax}{1+x}+\frac{1}{2}\bigg)\bigg)^2
+\bigg(\frac{1}{n_2+1}\bigg(-y+\frac{ay}{1+y}+\frac{1}{2}\bigg)\bigg)^2}\,\,  \bigg)\\
&\leq&M\bigg\{\omega_2\bigg(f;\sqrt{M_{n_1,n_2}(x,y)\bigg)}+\mbox{min}\{1,M_{n_1,n_2}(x,y)\}||f||_{C_B^2(I)}\bigg\}\\
&&+\omega\bigg(f;\sqrt{\bigg(\frac{1}{n_1+1}\bigg(-x+\frac{ax}{1+x}+\frac{1}{2}\bigg)\bigg)^2
+\bigg(\frac{1}{n_2+1}\bigg(-y+\frac{ay}{1+y}+\frac{1}{2}\bigg)\bigg)^2}\,\,  \bigg)
\end{eqnarray*}
holds. The constant $M\geq0$ is independent of $f$ and $M_{n_1,n_2}(x,y),$\\ where $M_{n_1,n_2}(x,y)=\dfrac{\{\delta_{n_1}^a(x)\}^2}{n_1+1}+\dfrac{\{\delta_{n_2}^a(y)\}^2}{n_2+1}.$
\end{thm}
\begin{proof}
We define the auxiliary operators as follows:
\begin{eqnarray}\label{pk1}
\overline K_{n_1,n_2}^a(f;x,y)=K_{n_1,n_2}^a(f;x,y)-f\bigg(\frac{1}{n_1+1}\bigg(n_1 x
+\frac{ax}{1+x}+\frac{1}{2}\bigg),\frac{1}{n_2+1}\bigg(n_2 y
+\frac{ay}{1+y}+\frac{1}{2}\bigg)\bigg)+ f(x,y).
\end{eqnarray}
Then, from Lemma \ref{l4}, we have\\
$\overline K_{n_1,n_2}^a((u-x);x,y)=0$  and $\overline K_{n_1,n_2}^a((v-y);x,y)=0.$\\
Let $g\in C_{B}^2(I)$ and $(u,v)\in I.$ Using the Taylor's theorem, we have
\begin{eqnarray}\label{mg1}
g(u,v)-g(x,y)&=&\frac{\partial g(x,y)}{\partial x}(u-x)+\int_{x}^u(u-\alpha)\frac{\partial^2 g(\alpha,y)}{\partial \alpha^2}d\alpha+\frac{\partial g(x,y)}{\partial y}(v-y)\nonumber\\&&+\int_{y}^v(v-\beta)\frac{ \partial^2 g(x,\beta)}{\partial \beta^2}d\beta.
\end{eqnarray}
Operating $\overline K_{n_1,n_2}^a$ on both sides of (\ref{mg1}), we get
\begin{eqnarray*}
\overline K_{n_1,n_2}^a(g;x,y)-g(x,y)&=&\overline K_{n_1,n_2}^a\bigg( \int_{x}^u(u-\alpha)\frac{\partial^2 g(\alpha,y)}{\partial \alpha^2}d\alpha;x,y\bigg)\\&&+\overline K_{n_1,n_2}^a\bigg(\int_{y}^v(v-\beta)\frac{ \partial^2 g(x,\beta)}{\partial \beta^2}d\beta;x,y\bigg)\\
&=&K_{n_1,n_2}^a\bigg(\int_{x}^u(u-\alpha)\frac{\partial^2 g(\alpha,y)}{\partial \alpha^2}d\alpha;x,y\bigg)\\&&-\int_{x}^{\frac{1}{n_1+1}\bigg(n_1 x
+\frac{ax}{1+x}+\frac{1}{2}\bigg)}\bigg(\frac{1}{n_1+1}\bigg(n_1 x
+\frac{ax}{1+x}+\frac{1}{2}\bigg)-\alpha\bigg)\frac{\partial^2 g(\alpha,y)}{\partial \alpha^2}d\alpha\\&&+K_{n_1,n_2}^a
\bigg(\int_{y}^v(v-\beta)\frac{ \partial^2 g(x,\beta)}{\partial \beta^2}d\beta;x,y\bigg)\\&&-\int_{y}^{\frac{1}{n_2+1}\bigg(n_2 y
+\frac{ay}{1+y}+\frac{1}{2}\bigg)}\bigg(\frac{1}{n_2+1}\bigg(n_2 y
+\frac{ay}{1+y}+\frac{1}{2}\bigg)-\beta\bigg)\frac{\partial^2 g(x,\beta)}{\partial \beta^2}d\beta.
\end{eqnarray*}
Hence,
\begin{eqnarray*}
|\overline K_{n_1,n_2}^a(g;x,y)-g(x,y)|&\leq &K_{n_1,n_2}^a\bigg(\bigg|\int_{x}^u|u-\alpha| \bigg|\frac{\partial^2 g(\alpha,y)}{\partial \alpha^2}\bigg|d\alpha\bigg|;x,y\bigg)\\&&+\bigg|\int_{x}^{\frac{1}{n_1+1}\bigg(n_1 x
+\frac{ax}{1+x}+\frac{1}{2}\bigg)}\bigg|{\frac{1}{n_1+1}\bigg(n_1 x
+\frac{ax}{1+x}+\frac{1}{2}\bigg)}-\alpha\bigg|\bigg|\frac{\partial^2 g(\alpha,y)}{\partial \alpha^2}\bigg|d\alpha\bigg|\\&&+K_{n_1,n_2}^a\bigg(\bigg|\int_{y}^v|v-\beta|\bigg|\frac{\partial^2 g(x,\beta)}{\partial \beta^2}\bigg|d\beta\bigg|;x,y\bigg)\\&&+\bigg|\int_{y}^{\frac{1}{n_2+1}\bigg(n_2 y
+\frac{ay}{1+y}+\frac{1}{2}\bigg)}\bigg|{\frac{1}{n_2+1}\bigg(n_2 y
+\frac{ay}{1+y}+\frac{1}{2}\bigg)}-\beta\bigg|\bigg|\frac{\partial^2 g(x,\beta)}{\partial \beta^2}\bigg|d\beta\bigg|\\
&\leq&\bigg\{K_{n_1,n_2}^a((u-x)^2;x,y)+\bigg(\frac{1}{n_2+1}\bigg(n_2 y
+\frac{ay}{1+y}+\frac{1}{2}\bigg)-x\bigg)^2\bigg\}||g||_{C_{B}^2(I)}\\&&+
\bigg\{K_{n_1,n_2}^a((v-y)^2;x,y)+\bigg(\frac{1}{n_2+1}\bigg(n_2 y
+\frac{ay}{1+y}+\frac{1}{2}\bigg)-y\bigg)^2\bigg\}||g||_{C_{B}^2(I)}\\
&\leq&\bigg\{\frac{1}{n_1+1}\bigg(\phi(x)+\frac{(1+a)^2}{n_1+1}\bigg)+\bigg(\frac{1}{n_1+1}
\bigg(-x+\frac{ax}{1+x}+\frac{1}{2}\bigg)\bigg)^2\\&&+\frac{1}{n_2+1}\bigg(\phi(y)
+\frac{(1+a)^2}{n_2+1}\bigg)+\bigg(\frac{1}{n_2+1}\bigg(-y+\frac{ay}{1+y}+\frac{1}{2}\bigg)\bigg)^2\bigg\}||g||_{C_{B}^2(I)}.
\end{eqnarray*}
Thus, we get
\begin{eqnarray*}
|\overline K_{n_1,n_2}^a(g;x,y)-g(x,y)|&\leq&\bigg\{\frac{2}{n_1+1}\bigg(\phi(x)+\frac{(1+a)^2}{n_1+1}\bigg)
+\frac{2}{n_2+1}\bigg(\phi(y)+\frac{(1+a)^2}{n_2+1}\bigg)\bigg\}||g||_{C_{B}^2(I)}.
\end{eqnarray*}
Also,
\begin{eqnarray}\label{mg2} |\overline K_{n_1,n_2}^a(f;x,y)|&\leq&|K_{n_1,n_2}^a(f;x,y)|+\bigg|f\bigg(\frac{1}{n_1+1}\bigg(n_1 x
+\frac{ax}{1+x}+\frac{1}{2}\bigg),\frac{1}{n_2+1}\bigg(n_2 y+\frac{ay}{1+y}+\frac{1}{2}\bigg)\bigg)\bigg|+ |f(x,y)|\nonumber\\
&\leq&3||f||_{C_{B}(I)}.
\end{eqnarray}
Now, from equation (\ref{mg2}), we get
\begin{eqnarray*}
|K_{n_1,n_2}^a(f;x,y)-f(x,y)|&\leq&|\overline K_{n_1,n_2}^a(f-g;x,y)|+|\overline K_{n_1,n_2}^a(g;x,y)-g(x,y)|+|g(x,y)-f(x,y)|\\&&+\bigg|f\bigg(\frac{1}{n_1+1}\bigg(n_1 x
+\frac{ax}{1+x}+\frac{1}{2}\bigg),\frac{1}{n_2+1}\bigg(n_2 y+\frac{ay}{1+y}+\frac{1}{2}\bigg)\bigg)-f(x,y)\bigg|\\
&\leq&3||f-g||_{C_B(I)}+||f-g||_{C_B(I)}+|\overline K_{n_1,n_2}^a(g;x,y)-g(x,y)|\\&&+\bigg|f\bigg(\frac{1}{n_1+1}\bigg(n_1 x
+\frac{ax}{1+x}+\frac{1}{2}\bigg),\frac{1}{n_2+1}\bigg(n_2 y+\frac{ay}{1+y}+\frac{1}{2}\bigg)\bigg)-f(x,y)\bigg|\\
&\leq&4||f-g||_{C_B(I)}+\bigg\{\frac{2}{n_1+1}\bigg(\phi(x)+\frac{(1+a)^2}{n_1+1}\bigg)
+\frac{2}{n_2+1}\bigg(\phi(y)+\frac{(1+a)^2}{n_2+1}\bigg)\bigg\}||g||_{C_{B}^2(I)}\\&&+\bigg|f\bigg(\frac{1}{n_1+1}\bigg(n_1 x
+\frac{ax}{1+x}+\frac{1}{2}\bigg),\frac{1}{n_2+1}\bigg(n_2 y+\frac{ay}{1+y}+\frac{1}{2}\bigg)\bigg)-f(x,y)\bigg|\\
&\leq&\bigg(4||f-g||_{C_B(I)}+2M_{n_1,n_2}(x,y)||g||_{C_{B}^2(I)}\bigg)\\
&&+\omega\bigg(f;\sqrt{\bigg(\frac{1}{n_1+1}\bigg(-x+\frac{ax}{1+x}+\frac{1}{2}\bigg)\bigg)^2
+\bigg(\frac{1}{n_2+1}\bigg(-y+\frac{ay}{1+y}+\frac{1}{2}\bigg)\bigg)^2} \,\,  \bigg).
\end{eqnarray*}
Taking the infimum  on the right hand side over all $g\in C_{B}^2(I)$ and using (\ref{fr1}), we obtain
\begin{eqnarray*}
|K_{n_1,n_2}^a(f;x,y)-f(x,y)|&\leq&4\mathcal{K}(f;M_{n_1,n_2}(x,y))\\&&+\omega\bigg(f;\sqrt{\bigg(\frac{1}{n_1+1}\bigg(-x+\frac{ax}{1+x}+\frac{1}{2}\bigg)\bigg)^2
+\bigg(\frac{1}{n_2+1}\bigg(-y+\frac{ay}{1+y}+\frac{1}{2}\bigg)\bigg)^2}\,\,  \bigg)\\
&\leq&M\bigg\{\omega_2\bigg(f;\sqrt{M_{n_1,n_2}(x,y)\bigg)}+\mbox{min}\{1,M_{n_1,n_2}(x,y)\}||f||_{C_B^2(I)}\bigg\}\\
&&+\omega\bigg(f;\sqrt{\bigg(\frac{1}{n_1+1}\bigg(-x+\frac{ax}{1+x}+\frac{1}{2}\bigg)\bigg)^2
+\bigg(\frac{1}{n_2+1}\bigg(-y+\frac{ay}{1+y}+\frac{1}{2}\bigg)\bigg)^2}\,\,  \bigg).
\end{eqnarray*}
Hence, the proof is completed.
\end{proof}

\subsection{Rate of convergence of bivariate operators}
\begin{thm}\label{thm5}
Suppose that $f\in C_{\gamma_1, \gamma_2}^{1}(I)$ with $\gamma_1, \gamma_2 \in \mathbb{N}_0$ then there exist a positive constant $M_5(\gamma_1, \gamma_2)$ such that for all $(x,y)\in I$ and $n_1,n_2 \in \mathbb{N}$
\begin{eqnarray*}
\omega_{\gamma_1,\gamma_2}(x,y)|K_{n_1,n_2}^a(f;x,y)-f(x,y)|\leq M_5(\gamma_1, \gamma_2)\bigg\{\parallel f_x^{\prime}\parallel_{\gamma_1, \gamma_2}\frac{\delta_{n_1}^a(x)}{\sqrt{n_1+1}}+\parallel f_y^{\prime}\parallel_{\gamma_1, \gamma_2}\frac{\delta_{n_2}^a(y)}{\sqrt{n_2+1}}\bigg\}.
\end{eqnarray*}
\end{thm}
\begin{proof}
Let $(x,y)\in I$ be a fixed point. Then, we have
\begin{eqnarray}\label{p3}
f(t,z)-f(x,y)=\int_{x}^tf_u^{\prime}(u,z)du+\int_{y}^zf_v^{\prime}(x,v)dv\,\,\,\,\, \mbox{for}\,\,\, (t,z)\in I\nonumber\\
K_{n_1,n_2}^a(f(t,z);x,y)-f(x,y) = K_{n_1,n_2}^a\bigg(\int_{x}^tf_u^{\prime}(u,z)du;x,y\bigg)+K_{n_1,n_2}^a\bigg(\int_{y}^zf_v^{\prime}(x,v)dv;x,y\bigg).
\end{eqnarray}
Now, by using (\ref{m1}), we get
\begin{eqnarray*}
\bigg|\int_{x}^tf_u^{\prime}(u,z)du\bigg|\leq \parallel f_x^{\prime}\parallel_{\gamma_1,\gamma_2}
\bigg|\int_x^t\frac{du}{\omega_{\gamma_1,\gamma_2}(u,z)}\bigg|\leq \parallel f_x^{\prime}
\parallel_{\gamma_1,\gamma_2}\bigg(\frac{1}{\omega_{\gamma_1,\gamma_2}(t,z)}
+\frac{1}{\omega_{\gamma_1,\gamma_2}(x,z)}\bigg)|t-x|,
\end{eqnarray*}
and analogously
\begin{eqnarray*}
\bigg|\int_{y}^zf_v^{\prime}(x,v)dv\bigg|\leq \parallel f_y^{\prime}
\parallel_{\gamma_1,\gamma_2}\bigg(\frac{1}{\omega_{\gamma_1,\gamma_2}(x,z)}
+\frac{1}{\omega_{\gamma_1,\gamma_2}(x,y)}\bigg)|z-y|.
\end{eqnarray*}
By using these inequalities and from (\ref{p1}), we obtain for $n_1, n_2\in \mathbb{N}$\\
\noindent
$\omega_{\gamma_1,\gamma_2}(x,y)\bigg|K_{n_1,n_2}^a\bigg(\displaystyle
\int_{x}^t f_u^{\prime}(u,z)du;x,y\bigg)\bigg|$
\begin{eqnarray}
&\leq& \omega_{\gamma_1,\gamma_2}(x,y)K_{n_1,n_2}^a\bigg
(\bigg|\int_{x}^tf_u^{\prime}(u,z)du\bigg|;x,y\bigg)\nonumber\\&\leq&\parallel
f_x^{\prime}\parallel_{\gamma_1,\gamma_2}\omega_{\gamma_1,\gamma_2}(x,y)
\bigg\{K_{n_1,n_2}^a\bigg(\frac{|t-x|}{\omega_{\gamma_1,\gamma_2}(t,z)};x,y\bigg)
+K_{n_1,n_2}^a\bigg(\frac{|t-x|}{\omega_{\gamma_1,\gamma_2}(x,z)};x,y\bigg)\bigg\}\nonumber\\
&=&\parallel f_x^{\prime}\parallel_{\gamma_1,\gamma_2}\omega_{\gamma_2}(y)K_{n_2}^a\bigg(\frac{1}
{\omega_{\gamma_2}(z)};y\bigg)\bigg\{\omega_{\gamma_1}(x)K_{n_1}^a\bigg(\frac{|t-x|}
{\omega_{\gamma_1}(t)};x\bigg)+K_{n_1}^a(|t-x|;x)\bigg\},
\end{eqnarray}
and analogously\\
\noindent
$\omega_{\gamma_1,\gamma_2}(x,y)\bigg|K_{n_1,n_2}^a\bigg(\displaystyle
\int_{y}^z f_v^{\prime}(x,v)dv;x,y\bigg)\bigg|$
\begin{eqnarray}
\leq \parallel f_y^{\prime}\parallel_{\gamma_1,\gamma_2}
\bigg\{\omega_{\gamma_2}(y)K_{n_2}^a\bigg(\frac{|z-y|}
{\omega_{\gamma_2}(z)};y\bigg)+K_{n_2}^a(|z-y|;y)\bigg\}.
\end{eqnarray}
By the H$\ddot{o}$lder inequality and Remark \ref{r1}, we get for $n_1\in \mathbb{N}$
\begin{eqnarray}
K_{n_1}^a(|t-x|;x)\leq \{K_{n_1}^a((t-x)^2;x)K_{n_1}^a(1;x)\}^{1/2} \leq \frac{\delta_{n_1}^a(x)}{\sqrt{n_1+1}}
\end{eqnarray}
and\\
\noindent
$\omega_{\gamma_1}(x)K_{n_1}^a \bigg(\frac{|t-x|}{\omega_{\gamma_1}(t)};x\bigg)$
\begin{eqnarray}
\leq \omega_{\gamma_1}(x)\bigg\{K_{n_1}^a \bigg(\frac{(t-x)^2}{\omega_{\gamma_1}(t)};x\bigg)
K_{n_1}^a \bigg(\frac{1}{\omega_{\gamma_1}(t)};x\bigg)\bigg\}^{1/2}\leq M_6(\gamma_1) \frac{\delta_{n_1}^a(x)}{\sqrt{n_1+1}}, \mbox{in view of Lemma \ref{l7}}.
\end{eqnarray}
Analogously for $n_2\in\mathbb{N},$ we have
\begin{eqnarray}
K_{n_2}^a(|z-y|;y)\leq \frac{\delta_{n_2}^a(y)}{\sqrt{n_2+1}},
\end{eqnarray}
and
\begin{eqnarray}\label{m7}
\omega_{\gamma_2}(y)K_{n_2}^a \bigg(\frac{|z-y|}{\omega_{\gamma_2}(z)};y\bigg)\leq M_7(\gamma_2)\frac{\delta_{n_2}^a(y)}{\sqrt{n_2+1}}.
\end{eqnarray}
From equations (\ref{p3})-(\ref{m7}), we obtain\\
$\omega_{\gamma_1,\gamma_2}(x,y)\mid K_{n_1,n_2}^a(f(t,z);x,y)-f(x,y)
\mid\leq M_8(\gamma_1,\gamma_2)\bigg\{\parallel f_x^{\prime}\parallel_{\gamma_1,\gamma_2}
\dfrac{\delta_{n_1}^a(x)}{\sqrt{n_1+1}}+\parallel f_y^{\prime}\parallel_{\gamma_1,\gamma_2}\dfrac{\delta_{n_2}^a(y)}{\sqrt{n_2+1}}\bigg\},$\\
for all $n_1,n_2\in \mathbb{N}.$ Thus the proof is completed.
\end{proof}

\begin{thm}\label{thm7}
Suppose that $f\in C_{\gamma_1,\gamma_2}(I)$ with some $\gamma_1,\gamma_2\in \mathbb{N}_0.$ Then there exists a positive constant $M_9(\gamma_1,\gamma_2)$ such that
\begin{eqnarray*}
\omega_{\gamma_1,\gamma_2}(x,y)\mid K_{n,n}^a(f(t,z);x,y)-f(x,y)\mid\leq M_9(\gamma_1,\gamma_2)\omega\bigg(f;C_{\gamma_1,\gamma_2};\frac{\delta_{n_1}^a(x)}{\sqrt{n_1+1}},\frac{\delta_{n_2}^a(y)}{\sqrt{n_2+1}}\bigg),
\end{eqnarray*}
for all $(x,y)\in I$ and $n_1,n_2\in \mathbb{N}.$
\end{thm}
\begin{proof}
Let $f_{h,\delta}$ be the Steklov function of $f\in C_{\gamma_1,\gamma_2}(I),$ defined by the formula
\begin{eqnarray}\label{p4}
f_{h,\delta}(x,y):= \frac{1}{h \delta}\int_{0}^{h} du\int_{0}^\delta f(x+u,y+v)dv,
\end{eqnarray}
$(x,y)\in I$ and $h,\delta\in\mathbb{R}_{+}.$ From (\ref{p4}) it follows that
\begin{eqnarray*}
f_{h,\delta}(x,y)-f(x,y)&=&\frac{1}{h \delta}\int_{0}^{h} du\int_{0}^\delta \Delta_{u,v}f(x,y)dv,\\
\frac{\partial}{\partial x}  f_{h,\delta}(x,y)&=&\frac{1}{h \delta} \int_{0}^\delta \Delta_{h,0}f(x,y+v)dv,\\
&=& \frac{1}{h \delta} \int_{0}^\delta (\Delta_{h,v}f(x,y)-\Delta_{0,v}f(x,y))dv,\\
\frac{\partial}{\partial y}  f_{h,\delta}(x,y)&=&\frac{1}{h \delta}\int_{0}^{h}\Delta_{0,\delta}f(x+u,y)du\\
&=&\frac{1}{h \delta} \int_{0}^\delta (\Delta_{u,\delta}f(x,y)-\Delta_{u,0}f(x,y))du.
\end{eqnarray*}
Thus, from (\ref{m1}) and (\ref{m2}) we obtain
\begin{eqnarray}\label{m3}
\parallel f_{h,\delta}-f\parallel_{\gamma_1,\gamma_2}&\leq& \omega(f, C_{\gamma_1,\gamma_2}; h,\delta),
\end{eqnarray}
\begin{eqnarray}\label{m4}
\bigg\| \frac{\partial f_{h,\delta}}{\partial x}\bigg\|_{\gamma_1,\gamma_2}&\leq& 2 h^{-1}\omega(f, C_{\gamma_1,\gamma_2}; h,\delta),
\end{eqnarray}
\begin{eqnarray}\label{m5}
\bigg\| \frac{\partial f_{h,\delta}}{\partial y}\bigg\|_{\gamma_1,\gamma_2}&\leq& 2 \delta^{-1}\omega(f, C_{\gamma_1,\gamma_2}; h,\delta).
\end{eqnarray}
For $h,\delta\in \mathbb{R}_+,$ we can write
\begin{eqnarray}\label{m6}
\omega_{\gamma_1,\gamma_2}(x,y)\mid K_{n_1,n_2}^a(f(t,z);x,y)-f(x,y)\mid&\leq&\omega_{\gamma_1,\gamma_2}(x,y)\{K_{n_1,n_2}^a(f(t,z)-f_{h,\delta}(t,z);x,y)\nonumber\\&&
+\mid K_{n_1,n_2}^a(f_{h,\delta}(t,z);x,y)-f_{h,\delta}(x,y)\mid\\&&
+\mid f_{h,\delta}(x,y)-f(x,y)\mid\}:=R_1+R_2+R_3\nonumber.
\end{eqnarray}
By (\ref{m1}), Lemma \ref{l5} and (\ref{m3}) it follows that
\begin{eqnarray*}
R_1\leq\parallel K_{n,n}^a(f-f_{h,\delta;.,.})  \parallel_{\gamma_1,\gamma_2}&\leq& M_{10}(\gamma_1,\gamma_2)\|f-f_{h,\delta}\|_{\gamma_1,\gamma_2}\\
&\leq& M_{10}(\gamma_1,\gamma_2)\omega(f,C_{\gamma_1,\gamma_2};h,\delta),
\end{eqnarray*}
and
$$R_3\leq\omega(f,C_{\gamma_1,\gamma_2};h,\delta).$$\\
By using Theorem \ref{thm5} and (\ref{m4}) and (\ref{m5}), we get
\begin{eqnarray*}
R_2&\leq &M_{11}(\gamma_1,\gamma_2)\bigg\{\bigg\| \frac{\partial f_{h,\delta}}{\partial x}\bigg\|_{\gamma_1,\gamma_2}\frac{\delta_{n_1}^a(x)}{\sqrt{n_1+1}}+\bigg\| \frac{\partial f_{h,\delta}}{\partial y}\bigg\|_{\gamma_1,\gamma_2}\frac{\delta_{n_2}^a(y)}{\sqrt{n_2+1}}\bigg\}\\
&\leq&2M_{11}(\gamma_1,\gamma_2)\omega(f,C_{\gamma_1,\gamma_2};h,\delta)\bigg\{h^{-1}
\frac{\delta_{n_1}^a(x)}{\sqrt{n_1+1}}+\delta^{-1}\frac{\delta_{n_2}^a(y)}{\sqrt{n_2+1}}\bigg\}.
\end{eqnarray*}
Consequently, we drive from (\ref{m6})
\begin{eqnarray*}
\omega_{\gamma_1,\gamma_2}(x,y)\mid K_{n_1,n_2}^a(f(t,z);x,y)-f(x,y)\mid\leq M_{12}(\gamma_1,\gamma_2)\omega(f,C_{\gamma_1,\gamma_2};h,\delta)\bigg\{1+h^{-1}
\frac{\delta_{n_1}^a(x)}{\sqrt{n_1+1}}+\delta^{-1}\frac{\delta_{n_2}^a(y)}{\sqrt{n_2+1}}\bigg\}
\end{eqnarray*}
for all $(x,y)\in I,n_1,n_2\in\mathbb{N}$ and $h,\delta\in \mathbb{R}_+.$ On choosing $h=\dfrac{\delta_{n_1}^a(x)}{\sqrt{n_1+1}}$ and $\delta=\dfrac{\delta_{n_2}^a(y)}{\sqrt{n_2+1}},$ we immediately obtain the required result.
\end{proof}

As a consequence of Theorem \ref{thm7}, we have
\begin{thm}\label{thm8}
Let $f\in C_{\gamma_1,\gamma_2}(I)$ with some $\gamma_1,\gamma_2\in \mathbb{N}_0.$ Then for every $(x,y)\in I,$
\begin{eqnarray*}
\displaystyle \lim_{n_1, n_2\rightarrow \infty}K_{n_1,n_2}^a(f;x,y)= f(x,y).
\end{eqnarray*}
\end{thm}

\begin{thm}\label{t9}{\bf{(Voronovskaja type theorem)}}
Let $f\in C_{\gamma_1,\gamma_2}^2(I).$ Then for every $(x,y)\in I,$
\begin{eqnarray*}
\displaystyle\lim_{n\rightarrow \infty}n\{K_{n,n}^a(f;x,y)-f(x,y)\}
&=&\bigg(-x+\frac{ax}{1+x}+\frac{1}{2}\bigg)f_x(x,y)+\bigg(-y+\frac{ay}
{1+y}+\frac{1}{2}\bigg)f_y(x,y)\\&&+\frac{x}{2}(x+2)f_{xx}(x,y)
+\frac{y}{2}(y+2)f_{yy}(x,y).
\end{eqnarray*}
\end{thm}
\begin{proof}
Let $(x,y)\in I$ be fixed. By Taylor formula, we have
\begin{eqnarray*}
f(u,v)&=&f(x,y)+f_{x}(x,y)(u-x)+f_{y}(x,y)(v-y)\\&&+\frac{1}{2}\{f_{xx}
(x,y)(u-x)^2+2f_{xy}(x,y)(u-x)(v-y)+f_{yy}
(x,y)(v-y)^2\}\\&&+ \psi(u,v;x,y)\sqrt{(u-x)^4+(v-y)^4},
\end{eqnarray*}
where $\psi(.,.;x,y)\equiv \psi(.,.)\in C_{\gamma_1,\gamma_2}(I)$ and $\psi(x,y)=0.$ Thus, we get
\begin{eqnarray}\label{e6}
K_{n,n}^a(f(u,v);x,y)&=&f(x,y)+f_{x}(x,y)K_{n,n}^a(u-x;x)+f_{y}(x,y)
K_{n,n}^a(v-y;y)\nonumber\\&&+\frac{1}{2}\{f_{xx}(x,y)K_{n,n}^a((u-x)^2;x)
+2f_{xy}(x,y)K_{n,n}^a(u-x;x)K_{n,n}^a(v-y;y)\nonumber\\&&+f_{yy}
(x,y)K_{n,n}^a((v-y)^2;y)\}+K_{n,n}^a(\psi(u,v)\sqrt{(u-x)^4+(v-y)^4};x,y).\nonumber\\
\end{eqnarray}
Applying the H$\ddot{o}$lder inequality, we have
\begin{eqnarray*}
|K_{n,n}^a(\psi(u,v)\sqrt{(u-x)^4+(v-y)^4};x,y)|&\leq& \{K_{n,n}^a(\psi^2(u,v);x,y)\}
^{1/2}\{K_{n,n}^a((u-x)^4+(v-y)^4;x,y)\}^{1/2}\\
&\leq& \{K_{n,n}^a(\psi^2(u,v);x,y)\}^{1/2}\{K_n^a((u-x)^4;x)+K_n^a((v-y)^4;y)\}^{1/2}.
\end{eqnarray*}
By Theorem \ref{thm8}
\begin{eqnarray*}
\displaystyle \lim_{n\rightarrow\infty}K_{n,n}^a(\psi^2(u,v);x,y)
=\psi^2(x,y)=0,
\end{eqnarray*}
and from Lemma \ref{lm1} (iii), for each $(x,y)\in I, K_n^a((u-x)^4;x)=
O(\frac{1}{n^2})$ and $K_n^a((v-y)^4;y)=O(\frac{1}{n^2}).$\\
Hence
\begin{eqnarray}\label{e5}
\displaystyle \lim_{n\rightarrow\infty}n K_{n,n}^a(\psi(u,v)\sqrt{(u-x)^4+(v-y)^4};x,y)=0.
\end{eqnarray}
By combining (\ref{e6}) and (\ref{e5}), we obtain the desired result.
\end{proof}

\subsection{Simultaneous approximation}
\begin{thm}\label{t10}
Let $f\in C_{\gamma_1,\gamma_2}^1(I).$ Then for every $(x,y)\in \mathbb{R}_+^2=\mathbb{R}_+\times \mathbb{R}_+,$
\begin{eqnarray}\label{e11}
\displaystyle \lim_{n\rightarrow\infty}\bigg(\frac{\partial}{\partial \omega}
K_{n,n}^a(f;\omega,y)\bigg)_{\omega=x}&=&\frac{\partial f}{\partial x}(x,y),
\end{eqnarray}
\begin{eqnarray}\label{e12}
\displaystyle \lim_{n\rightarrow\infty}\bigg(\frac{\partial}{\partial \nu}
K_{n,n}^a(f;x,\nu)\bigg)_{\nu=y}&=&\frac{\partial f}{\partial y}(x,y).
\end{eqnarray}
\end{thm}
\begin{proof}
We shall prove only (\ref{e11}) because the proof of (\ref{e12}) is identical.\\
By the Taylor formula for $f\in C_{\gamma_1,\gamma_2}^1(I),$ we have
\begin{eqnarray*}
f(u,v)=f(x,y)+f_x(x,y)(u-x)+f_y(x,y)(v-y)+\psi(u,v;x,y)
\sqrt{(u-x)^2+(v-y)^2}   \,\,\, \mbox{for}\,\, (u,v)\in I,
\end{eqnarray*}
where $\psi(u,v;x,y)\equiv \psi(.,.)\in C_{\gamma_1, \gamma_2}(I)$ and $\psi(x,y)=0.$\\
Operating $K_{n,n}^a(.;.,y)$ to the above inequality and then by using
Lemma \ref{l4}, we get
\begin{eqnarray*}
\bigg(\frac{\partial}{\partial \omega}K_{n,n}^a(f(u,v);\omega,y)\bigg)
_{\omega=x}&=&f(x,y)\bigg(\frac{\partial}{\partial \omega}K_{n,n}^a(1;\omega,y)\bigg)
_{\omega=x}+f_x(x,y)\bigg(\frac{\partial}{\partial \omega}
K_{n,n}^a(u-x;\omega,y)\bigg)_{\omega=x}\\&&+f_y(x,y)\bigg(\frac{\partial}
{\partial \omega}K_{n,n}^a(v-y;\omega,y)\bigg)_{\omega=x}\\&&+\bigg(\frac{\partial}
{\partial \omega}K_{n,n}^a(\psi(u,v;x,y)\sqrt{(u-x)^2+(v-y)^2};\omega,y)\bigg)
_{\omega=x},\,\,\, \mbox{for}\,\, (u,v)\in I\\
&=&f_x(x,y)\bigg\{\frac{\partial}{\partial \omega}\bigg(\frac{1}{n+1}\bigg(n\omega
+\frac{a\omega}{1+\omega}+\frac{1}{2}\bigg)\bigg)\bigg\}_{\omega=x}\\&&+f_y(x,y)
\bigg\{\frac{\partial}{\partial \omega}\bigg(\frac{1}{n+1}\bigg(ny+\frac{ay}{1+y}
+\frac{1}{2}\bigg)\bigg)\bigg\}_{\omega=x}+E\,\mbox{(say)}.
\end{eqnarray*}
It is sufficient to prove that $E\rightarrow 0$ as $n\rightarrow \infty.$
\begin{eqnarray*}
E&=&(n+1)^2\sum_{k_1=0}^{\infty}\sum_{k_2=0}^{\infty}\bigg(\frac{\partial}
{\partial \omega}W_{n,n,k_1,k_2}^a(\omega,y)\bigg)_{\omega=x}\int_{\frac{k_2}{n+1}}
^{\frac{k_2+1}{n+1}}\int_{\frac{k_1}{n+1}}^{\frac{k_1+1}{n+1}}\psi(u,v)
\sqrt{(u-x)^2+(v-y)^2}du\,dv\\&=&(n+1)^2\sum_{k_1=0}^{\infty}\sum_{k_2=0}^{\infty}
\frac{\{(k_1-nx)(1+x)-ax\}}{x(1+x)^2}W_{n,n,k_1,k_2}^a(x,y)\int_{\frac{k_2}{n+1}}
^{\frac{k_2+1}{n+1}}\int_{\frac{k_1}{n+1}}^{\frac{k_1+1}{n+1}}\psi(u,v)
\sqrt{(u-x)^2+(v-y)^2}du\,dv\\&=&\frac{n(n+1)^2}{x(1+x)}\sum_{k_1=0}
^{\infty}\sum_{k_2=0}^{\infty}\bigg(\frac{k_1}{n}-x\bigg)W_{n,n,k_1,k_2}^a(x,y)
\int_{\frac{k_2}{n+1}}^{\frac{k_2+1}{n+1}}\int_{\frac{k_1}{n+1}}^{\frac{k_1+1}{n+1}}
\psi(u,v)\sqrt{(u-x)^2+(v-y)^2}du\,dv\\&&-\frac{a}{(1+x)^2}K_{n,n}^a(\psi(u,v)
\sqrt{(u-x)^2+(v-y)^2};\omega,y)\\&=&E_1+E_2,\,\,\mbox{say.}
\end{eqnarray*}
First, we estimate $E_1$ by using Schwarz inequality.
\begin{eqnarray*}
E_1&\leq& \frac{n}{x(1+x)}\bigg(\sum_{k_1=0}^{\infty}W_{n,k_1}^a(x)\bigg(
\frac{k_1}{n}-x\bigg)^2\bigg)^{1/2}\\&&\times\bigg((n+1)^2\sum_{k_1=0}^{\infty}
\sum_{k_2=0}^{\infty}W_{n,n,k_1,k_2}^a(x,y)\int_{\frac{k_2}{n+1}}^{\frac{k_2+1}{n+1}}\int_{\frac{k_1}{n+1}}
^{\frac{k_1+1}{n+1}}\psi^2(u,v)((u-x)^2+(v-y)^2) du\,dv\bigg)^{1/2}\\
&\leq&  \frac{n}{x(1+x)}\bigg(\sum_{k_1=0}^{\infty}W_{n,k_1}^a(x)\bigg(
\frac{k_1}{n}-x\bigg)^2\bigg)^{1/2}\{K_{n,n}^a(\psi^4(u,v);x,y)(K_{n}^a
((u-x)^4;x)\\&&+2K_{n}^a((u-x)^2;x)(K_{n}^a((v-y)^2;y)+K_{n}^a((v-y)^4;y))\}^{1/4}\\
|E_1|&\leq& M_{12}(x,y)\{K_{n,n}^a(\psi^4(u,v);x,y)\}^{1/4},\,\, \mbox{in view of Lemma \ref{lm1}}.
\end{eqnarray*}
From Theorem \ref{thm8}, we obtain
\begin{eqnarray*}
\lim_{n\rightarrow\infty} K_{n,n}^a(\psi^4(u,v);x,y)=\psi^4(x,y)=0, \,\, \mbox{for} \,\, (x,y)\in \mathbb{R}_+^2.
\end{eqnarray*}
To estimate $E_2,$ proceeding in a manner similar to the estimate of $E_1$, we get $E_2\rightarrow 0$ as $n\rightarrow \infty.$\\Combining the estimates of $E_1$ and $E_2,$ we it follows that $E\rightarrow 0$ as $n\rightarrow \infty.$ This completes the proof.
\end{proof}

Similarly, we can prove the following theorem :
\begin{thm}
Let $f\in C_{\gamma_1,\gamma_2}^3(I).$ Then for every $(x,y)\in \mathbb{R}_+^2,$ we have
\begin{eqnarray*}
\displaystyle \lim_{n\rightarrow\infty}n\bigg\{\bigg(\frac{\partial}{\partial \omega}
K_{n,n}^a(f;\omega,y)\bigg)_{\omega=x}-\frac{\partial f}{\partial x}(x,y)\bigg\}
&=&\bigg(-1+\frac{a}{(1+x)^2}\bigg)f_x(x,y)+\left(1+\frac{ax}{1+x}\right)f_{xx}(x,y)\\
&&+\bigg(-y+\frac{ay}{1+y}+\frac{1}{2}\bigg)f_{xy}(x,y)+\frac{y}{2}(1+y)f_{xyy}(x,y)\\&&
+\frac{x}{2}(1+x)f_{xxx}(x,y)
\end{eqnarray*}
and
\begin{eqnarray*}
\displaystyle \lim_{n\rightarrow\infty}n\bigg\{\bigg(\frac{\partial}{\partial \nu}
K_{n,n}^a(f;x,\nu)\bigg)_{\nu=y}-\frac{\partial f}{\partial y}(x,y)\bigg\}
&=&\bigg(-1+\frac{a}{(1+y)^2}\bigg)f_y(x,y)+\left(1+\frac{ay}{1+y}\right)f_{yy}(x,y)\\
&&+\bigg(-x+\frac{ax}{1+x}+\frac{1}{2}\bigg)f_{xy}(x,y)+\frac{x}{2}(1+x)f_{xxy}(x,y)
\\&&+\frac{y}{2}(1+y)f_{yyy}(x,y).
\end{eqnarray*}
\end{thm}

{\bf Acknowledgement}
The first author is thankful to the "Council of Scientific and Industrial Research" (Grant code: 09/143(0836)/2013-EMR-1)
India for financial support to carry out the above research work.

\end{document}